\documentclass[a4paper,11pt]{article}

\usepackage{amsmath}
\usepackage{amssymb}
\usepackage{amsthm}
\usepackage{graphicx}
\usepackage{amscd}
\usepackage{epic, eepic}
\usepackage{url}
\usepackage{color}
\usepackage[utf8]{inputenc} 
\usepackage{comment}
\usepackage{enumerate}   
\usepackage{todonotes}
\usepackage{graphicx}
\usepackage{epstopdf}
\usepackage{enumitem}
\usepackage{tikz}
\usetikzlibrary{calc,decorations.pathreplacing}
\usetikzlibrary{calc,arrows.meta}
\usepackage{parskip}
\usepackage[top=18.5mm, bottom=18.5mm, left=18.5mm, right=18.5mm]{geometry}
\usepackage[bookmarks=false,hidelinks]{hyperref}
\usepackage{bbm}

\usepackage{setspace}

\makeatletter
\renewenvironment{proof}[1][\proofname] {\par\pushQED{\qed}\normalfont\topsep6\p@\@plus6\p@\relax\trivlist\item[\hskip\labelsep\bfseries#1\@addpunct{.}]\ignorespaces}{\popQED\endtrivlist\@endpefalse}
\makeatother

\newtheorem{theorem}{\bf Theorem}[section]
\newtheorem{lemma}[theorem]{\bf Lemma}

\newtheorem{corollary}[theorem]{\bf Corollary}
\newtheorem{proposition}[theorem]{\bf Proposition}
\newtheorem{conjecture}[theorem]{\bf Conjecture}

\theoremstyle{definition}

\newtheorem{definition}[theorem]{\bf Definition}

\def\ex{\mathrm{ex}}

\title{A proof of the rational exponents conjecture}
\date{}

\author{David Conlon\footnote{Department of Mathematics, California Institute of Technology, Pasadena, CA 91125, USA. Email: {\tt dconlon@caltech.edu}.
Supported by NSF Award DMS-2348859.} \and Oliver Janzer\footnote{Institute of Mathematics, EPFL,  Lausanne, Switzerland. Email: {\tt oliver.janzer@epfl.ch}.} \and 
Tao Jiang\footnote{Department of Mathematics, Miami University, Oxford, OH 45056, USA.
Email: {\tt jiangt@miamioh.edu}.}
\and
Jaehoon Kim\footnote{Department of Mathematical Sciences, KAIST, South Korea. Email: {\tt jaehoon.kim@kaist.ac.kr}. Supported by the National Research Foundation of Korea (NRF) grant funded by the Korea government (MSIT) No. RS-2023-00210430. } 
\and
Joonkyung Lee\thanks{School of Computational Sciences, KIAS, Seoul, South Korea. Email: \texttt{joonkyunglee@kias.re.kr}. Supported by the National Research Foundation of Korea (NRF) grant MSIT NRF-2022R1C1C1010300 and KIAS Individual Grant CG112101.} 
\and
Hong Liu\thanks{Extremal Combinatorics and Probability Group (ECOPRO), Institute for Basic Science (IBS), Daejeon, South Korea. Email: \texttt{hongliu@ibs.re.kr}. Supported by the Institute for Basic Science (IBS-R029-C4).}
\and Sean Longbrake\thanks{Department of Mathematics, Emory University, Atlanta, United States. Email: \texttt{sean.longbrake@emory.edu}.}
\and Liana Yepremyan\thanks{Department of Mathematics, Emory University, Atlanta, United States. Email: \texttt{lyeprem@emory.edu}.
}}
\begin{document}

\maketitle

\begin{abstract}
    The rational exponents conjecture of Erd\H os and Simonovits asserts that for every rational $r\in [1,2]$ there exists a graph $H$ such that $\ex(n,H)=\Theta(n^r)$. We present a proof of this conjecture claimed by Epoch AI, explaining how it builds upon ideas and methods developed by many researchers.
\end{abstract}

\section{Introduction}

Given a natural number $n$ and a graph $H$, the extremal number $\ex(n,H)$ is the largest number of edges in a graph on $n$ vertices containing no copy of $H$. When $H$ is not bipartite, the value of $\ex(n,H)$ is reasonably well understood, with a result of Erd\H{o}s and Stone~\cite{erdosstone46}, augmented by an observation of Erd\H{o}s and Simonovits~\cite{erdossim66}, implying that
\[\ex(n,H) = \left(1 - \frac{1}{\chi(H)-1} + o(1)\right) \binom{n}{2},\]
where $\chi(H)$ is the chromatic number of $H$. 

When $H$ is bipartite, the Erd\H{o}s--Stone--Simonovits theorem only implies that $\ex(n, H) = o(n^2)$. However, a stronger estimate due to K\H{o}v\'ari, S\'os and Tur\'an~\cite{kst54} implies that for any bipartite graph $H$ there exist positive constants $c$ and $\delta$ such that $\ex(n, H) \le c n^{2 - \delta}$. The question that then arises is whether one can determine the best possible $\delta$ for a given $H$. This problem is still open even for some of the most basic bipartite graphs, such as the complete bipartite graph $K_{4,4}$ and the even cycle $C_8$. 

Given the difficulties of making progress with specific graphs, Erd\H{o}s and Simonovits (see, for example,~\cite{erdos81}) posed several questions about the general nature of the function $\ex(n,H)$. One such question asks whether, for every graph $H$, there exists a rational exponent $r$ such that $\ex(n,H) = \Theta(n^r)$. Conversely, they also asked if, for every rational $r \in [1,2]$, there exists a graph $H$ such that $\ex(n,H) = \Theta(n^r)$. It is this second conjecture that we will be concerned with here, referring to it as the rational exponents conjecture, though the phrase could easily be used for either direction.

Until a decade ago, no general mechanism was known for giving sharp lower bounds for extremal numbers of bipartite graphs. This changed with Bukh's introduction~\cite{bukh15} of the random algebraic method, which uses random polynomials over finite fields to construct extremal graphs. In its most general form, due to Bukh and Conlon~\cite{bukhconlon18}, this method applies to rooted bipartite graphs. Such a graph $F$ has certain designated vertices, called roots, and its $\ell$th power $F^\ell$ is obtained by gluing $\ell$ copies of $F$ together along the roots, with no additional overlaps. 
Assuming a certain balancedness condition (see Subsection~\ref{subsec:balanced}) and that $\ell$ is sufficiently large, the method gives a lower bound for the extremal number of $F^\ell$. This bound is known to be optimal for a wide class of graphs and is conjecturally optimal whenever $F$ is a tree (see Conjecture~\ref{conj:bukh-conlon}).

In fact, Bukh and Conlon showed that the same construction avoids the whole family $\mathcal{F}_\ell$ of graphs defined by gluing $\ell$ distinct copies of $F$ together along the roots, with additional overlaps now allowed. Using this result, Bukh and Conlon~\cite{bukhconlon18} were able to prove a variant of the rational exponents conjecture where we allow finite families rather than just single graphs. Note that, given a finite family of graphs $\mathcal{H}$, the extremal number $\ex(n,\mathcal{H})$ is the largest number of edges in a graph on $n$ vertices containing no copy of any graph from $\mathcal{H}$.

\begin{theorem}[Bukh--Conlon~\cite{bukhconlon18}]
For every rational $r \in [1,2]$, there exists a finite family of graphs $\mathcal{H}$ such that $\ex(n, \mathcal{H}) = \Theta(n^r)$.
\end{theorem}

The reason this does not resolve the full conjecture is that the natural upper bound argument does not suffice. When $F$ is a tree, it is simple to count unrooted copies of $F$ and then an averaging argument shows that there must be some fixed choice of the roots contained in many copies of $F$. However, such an elementary argument does not give us any control on how these different copies of $F$ intersect, resulting in a member of $\mathcal{F}_\ell$, but not necessarily a copy of $F^\ell$. The attempt to establish the full rational exponents conjecture has therefore centred around efforts to resolve this degeneracy issue, mainly focusing on the case where $F$ is a tree, which would already imply the rational exponents conjecture. This program has led to the development of a variety of interesting techniques, allowing members of the combinatorial community to prove the conjecture across a number of ranges~\cite{conjan22,conlon2021more,janzer2020extremal,jjm22,jly26,jmy22,jiangqiu20, jiangqiu23,kkl21}. To give some concrete instances, these efforts have resulted in a proof of the conjecture for all exponents~\cite{jiangqiu23} of the form $1 + a/b$ with $b > a^2$ and all exponents~\cite{conjan22} of the form $2 - a/b$ with $b \ge (a-1)^2$. 

In this paper, we wish to explain and situate a proof of the conjecture for all exponents claimed by Epoch AI~\cite{bloom26}, obtained during benchmarking for the recently released model GPT-6 Astra. 

\begin{theorem}
For every rational $r \in [1,2]$, there exists a graph $H$ such that $\ex(n, H) = \Theta(n^r)$.
\end{theorem}

As with many of the recent AI-driven breakthroughs, this advance draws heavily on ideas and methods developed by many researchers over several years. Concretely, in this case, one may point to four prominent sources:

\begin{itemize}

\item
The result of Bukh and Conlon~\cite{bukhconlon18}, which already establishes the required lower bound.

\item
Foundational work of Erd\H{o}s and Simonovits~\cite{erdossim70} implying, roughly speaking, that if $2 - a/b$ is an exponent realised by a large power of a balanced rooted bipartite graph, then so is $2 - a/(a+b)$. 

\item
A conjecture of Kang, Kim and Liu~\cite{kkl21} stating that if a graph $H$ has $\ex(n,H)=O(n^{1+c})$, then its $1$-subdivision $H'$, the graph obtained by adding an additional vertex to each edge, has $\ex(n,H')=O(n^{1+c/2})$. As shown by the same authors, together with the Erd\H{o}s--Simonovits result above and the Bukh--Conlon lower bound, this conjecture would allow one, starting from some easy exponents, to percolate to all possible exponents. 

\item
A technique developed by Conlon, Janzer and Lee~\cite{conlon2021more} which underlies the proof of the main technical result, a variant of the Kang--Kim--Liu subdivision conjecture. 

\end{itemize}

To state this technical result, given a connected bipartite graph $H$ with bipartition $(A,B)$ and a positive integer $k$, we write $H_k$ for the graph obtained by replacing each edge of $H$ with a path of length $k$ and adding two more vertices, one joined to all vertices in $A$ and one joined to all vertices in $B$.  Noting that the $k=1$ case is the true statement of the result of Erd\H os and Simonovits~\cite{erdossim70} mentioned above, the key technical result is now as follows.

\begin{theorem} \label{thm:recursive upper bound}
    Let $H$ be a connected bipartite graph with $\ex(n,H)=O(n^{\beta})$ for some $1\leq \beta\leq 2$. Then, for any positive integer $k$, $\ex(n,H_k)=O(n^{1+\frac{1}{k+2-\beta}})$.
\end{theorem}

This result then stands in for the Kang--Kim--Liu subdivision conjecture to allow percolation from an initial set of easy exponents, in this case just the single exponent $1$, to all exponents. The details of this reduction may be found in the next section before we prove Theorem~\ref{thm:recursive upper bound} in Section~\ref{sec:recursive}. We conclude with some further remarks on the problems that remain open.

\section{Reduction to Theorem \ref{thm:recursive upper bound}}
\label{sec:lowerbound}

In this section, we construct a family of balanced rooted bipartite
graphs and obtain lower bounds for the extremal numbers of their powers using the result of Bukh and Conlon~\cite{bukhconlon18}. We then show how this, together with Theorem~\ref{thm:recursive upper bound},  proves the rational exponents conjecture.

\subsection{Balanced rooted graphs} \label{subsec:balanced}

A \emph{rooted graph} is a pair $(H,R)$, where $H$ is a graph and
$\emptyset\neq R\subsetneq V(H)$. We call the vertices in $R$ the
\emph{roots} and write $H$ instead of $(H,R)$ when the root
set is clear from context. For $S\subseteq V(H)$, let $e_H(S)$ denote the number
of edges of $H$ incident with at least one vertex in $S$. For each
non-empty $S\subseteq V(H)$, let
$$
    \rho_H(S):=\frac{e_H(S)}{|S|},
    \qquad
    \rho(H):=\rho_H(V(H)\setminus R).
$$
We say that $H$ is \emph{balanced} if $\rho_H(S)\geq\rho(H)$ for
every non-empty set $S\subseteq V(H)\setminus R$. Note that edges
between roots are not relevant to either the rooted density or balancedness.

For a positive integer $\ell$, the \emph{$\ell$th power} $H^\ell$
is the graph obtained from $\ell$ disjoint copies of $H$ by
identifying the corresponding copies of each root. Thus, the copies
are vertex disjoint outside $R$.

The following result of Bukh and Conlon provides the lower bound
we need. We use the version recorded
in~\cite[Lemma~2.3]{kkl21}.

\begin{lemma}[Bukh--Conlon~\cite{bukhconlon18}]
\label{lem:lower bound}
    Let $H$ be a balanced rooted bipartite graph with $\rho(H)\geq1$.
    Then there exists a positive integer $\ell_0=\ell_0(H)$ such that,
    for every fixed $\ell\geq\ell_0$,
    $$
        \ex(n,H^\ell)=\Omega\bigl(n^{2-1/\rho(H)}\bigr).
    $$
\end{lemma}

Bukh and Conlon stated their result for balanced rooted
trees with independent roots, but, as observed
in~\cite{kkl21}, their proof also applies to
balanced rooted bipartite graphs with independent roots. We may even allow edges between roots. Indeed, for 
$H'=H-E(H[R])$ with the same root set $R$, we have $\rho(H')=\rho(H)$ and $H'$ is balanced if and only if $H$ is. 
Since $(H')^\ell\subseteq H^\ell$, we have
$$
    \ex(n,H^\ell)\geq\ex(n,(H')^\ell)
       =\Omega\bigl(n^{2-1/\rho(H)}\bigr).
$$

\subsection{The recursive construction}

Recall the graph $H_k$ defined in the introduction, where each edge
of $H$ is replaced by a path of length~$k$ and two new vertices are added. Suppose that $H$ is a
connected bipartite graph rooted on $R$ and let $x,y$ be the two
additional vertices of $H_k$. For each $e\in E(H)$, let $P_e$ be
the path replacing $e$. We consider $H_k$ as a rooted graph with root set
$$
    R_k:=R\cup\{x,y\}\cup
       \bigcup_{e\in E(H[R])} V(P_e).
$$
Thus, every subdivision vertex on an edge between two original roots is
also a root. None of the other subdivision vertices are roots. See Figure \ref{fig:rooted} for an illustration.

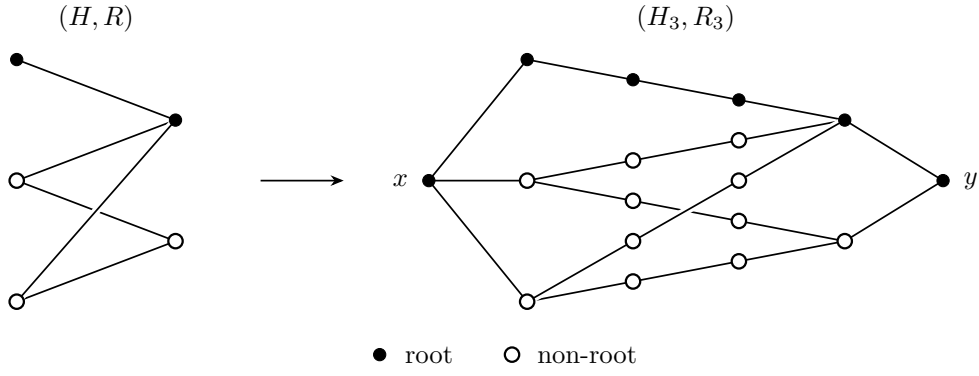
\begin{figure}
\centering
\begin{tikzpicture}[
  x=1cm,y=1cm,font=\small,
  edge/.style={draw=black,line width=.65pt},
  root/.style={circle,draw=black,fill=black,inner sep=0pt,minimum size=4.6pt},
  ordinary/.style={circle,draw=black,fill=white,line width=.85pt,
    inner sep=0pt,minimum size=5.4pt}
]

\coordinate (a1) at (0,1.60);
\coordinate (b1) at (2.10,.80);
\coordinate (a2) at (0,0);
\coordinate (b2) at (2.10,-.80);
\coordinate (a3) at (0,-1.60);
\draw[edge] (a1)--(b1)--(a2)--(b2)--(a3);
% A small white underlay distinguishes the crossing from a vertex.
\draw[edge,preaction={draw=white,line width=3pt}] (a3)--(b1);
\foreach \v in {a1,b1}{\node[root] at (\v) {};}
\foreach \v in {a2,b2,a3}{\node[ordinary] at (\v) {};}
\node at (1.05,2.15) {$(H,R)$};

% Construction arrow.
\draw[-{Stealth[length=5pt,width=4pt]},line width=.65pt]
  (3.225,0)--(4.325,0);

\begin{scope}[xshift=1.50cm]
\coordinate (A1) at (5.25,1.60);
\coordinate (B1) at (9.45,.80);
\coordinate (A2) at (5.25,0);
\coordinate (B2) at (9.45,-.80);
\coordinate (A3) at (5.25,-1.60);
\coordinate (X) at (3.95,0);
\coordinate (Y) at (10.75,0);

% The five new hub edges are single edges.
\foreach \v in {A1,A2,A3}{\draw[edge] (X)--(\v);}
\foreach \v in {B1,B2}{\draw[edge] (Y)--(\v);}

% The five original edges are each subdivided exactly twice.
\foreach \s/\t in {A1/B1,B1/A2,A2/B2,B2/A3}{
  \draw[edge] (\s)--(\t);
}
\draw[edge,preaction={draw=white,line width=3pt}] (A3)--(B1);
% Internal vertices on the edge between roots are roots.
\foreach \p in {0.333333,0.666667}{
  \node[root] at ($(A1)!\p!(B1)$) {};
}
% Internal vertices on the other four paths are non-roots.
\foreach \s/\t in {B1/A2,A2/B2,B2/A3,A3/B1}{
  \foreach \p in {0.333333,0.666667}{
    \node[ordinary] at ($(\s)!\p!(\t)$) {};
  }
}

% Original vertices and the two new roots.
\foreach \v in {A1,B1,X,Y}{\node[root] at (\v) {};}
\foreach \v in {A2,A3,B2}{\node[ordinary] at (\v) {};}
\node[left=4pt] at (X) {$x$};
\node[right=4pt] at (Y) {$y$};
\node at (7.35,2.15) {$(H_3,R_3)$};
\end{scope}

% Vertex convention.
\node[root] at (4.80,-2.30) {};
\node[anchor=west] at (5.00,-2.30) {root};
\node[ordinary] at (6.55,-2.30) {};
\node[anchor=west] at (6.75,-2.30) {non-root};

\end{tikzpicture}
\caption{The definition of the rooted graph $H_k$ for $k=3$.}
\label{fig:rooted}
\end{figure}

Observe that $H_k$ is connected and bipartite. If $k$ is odd, the
original vertices retain their place in the bipartition and each of $x,y$ is placed
opposite its neighbours. If $k$ is even, all the original vertices are placed in one part and $x,y$ are placed in the other. In both cases, the
internal vertices of each path $P_e$ are placed in an alternating way.

\begin{lemma}
\label{lem:recursive balancedness}
    Let $H$ be a balanced connected rooted bipartite graph with
    $\rho(H)\geq1$. Then, for every positive integer $k$, the graph
    $H_k$ rooted on $R_k$ is balanced with
    \begin{equation}\label{eq:rooted density}
        \rho(H_k)=\frac{k\rho(H)+1}{(k-1)\rho(H)+1}.
    \end{equation}
    Moreover, for every positive integer $\ell$, 
    \begin{equation}\label{eq:rooted powers}
        (H_k)^\ell\cong(H^\ell)_k.
    \end{equation}
\end{lemma}

\begin{proof}
    Let $U:=V(H)\setminus R$, $a:=|U|$ and $b:=e_H(U)$.
    As $\rho(H)=b/a$, it suffices to show that
    $\rho(H_k)=(kb+a)/((k-1)b+a)$. Each of the $b$ edges incident
    with $U$ gives $k-1$ new non-root vertices and $k$ edges incident
    with non-roots. Each vertex in $U$ is also adjacent to exactly
    one of $x,y$. Moreover, the paths replacing edges between roots contain
    only roots. Hence,
    $$
        |V(H_k)\setminus R_k|=a+(k-1)b,
        \qquad
        e_{H_k}(V(H_k)\setminus R_k)=kb+a.
    $$
    This proves~\eqref{eq:rooted density}.

    Let $S$ be a non-empty subset of $V(H_k)\setminus R_k$ and
    let $S_0:=S\cap U$. We claim that 
    $e_{H_k}(S)\geq\rho(H_k)|S|$. Let $\mathcal{E}$ be the set of
    edges $e\in E(H)$ for which $P_e$ meets $S$ and put
    $p:=|\mathcal{E}|$. Every edge incident with $S_0$ belongs to
    $\mathcal{E}$. Thus, as $H$ is balanced,
    $$
        p\geq e_H(S_0)\geq\rho(H)|S_0|,
    $$
    where the last inequality also holds when $S_0=\emptyset$.
    Moreover, every vertex in $S\setminus S_0$ lies internally on
    one of the $p$ paths corresponding to $\mathcal{E}$. Hence, 
    \begin{equation}\label{eq:nonroot set size}
        |S|\leq |S_0|+(k-1)p
          \leq\bigl(1/\rho(H)+k-1\bigr)p.
    \end{equation}

    For each $e=uv\in\mathcal{E}$, the path $P_e$ has at least
    $|S\cap(V(P_e)\setminus \{u,v\})|+1$ edges incident with $S$.
    Indeed, if $S$ contains an internal vertex of $P_e$, each
    maximal interval of such vertices has one more incident edge
    than the number of vertices in that interval. Otherwise, $S$ contains an
    endpoint of $P_e$ and hence an edge of $P_e$ is incident with $S$.
    Summing over $\mathcal{E}$ and adding the $|S_0|$ edges from
    $S_0$ to $x$ or $y$ gives
    $$
        e_{H_k}(S)\geq |S\setminus S_0|+p+|S_0|=|S|+p.
    $$
    Together with~\eqref{eq:nonroot set size}, this implies that
    $$
        \rho_{H_k}(S)\geq1+\frac{1}{k-1+1/\rho(H)}
        =\frac{k\rho(H)+1}{(k-1)\rho(H)+1}=\rho(H_k),
    $$
    so $H_k$ is balanced.

    To prove~\eqref{eq:rooted powers}, note that an edge incident
    with an original non-root has a replacement path in each
    of the $\ell$ copies in both constructions. An edge between
    two original roots has a single common replacement path, as all its
    internal vertices are roots of $H_k$. The vertices $x,y$ are
    also common roots and their incident edges agree in the two
    constructions. Therefore, $(H_k)^\ell\cong(H^\ell)_k$.
\end{proof}

\subsection{Obtaining all rational densities}

Let $\mathcal{F}$ be the smallest collection of rooted graphs
satisfying the following conditions:
\begin{itemize}
    \item $\mathcal{F}$ contains $K_2$ rooted on one endpoint and
    \item if $H\in\mathcal{F}$ and $k$ is a positive integer, then
    $H_k$ rooted on $R_k$ also belongs to $\mathcal{F}$.
\end{itemize}
The initial rooted graph is balanced with density $1$.
Lemma~\ref{lem:recursive balancedness} then shows that every
$H\in\mathcal{F}$ is balanced, connected and bipartite with
$\rho(H)\geq1$.

\begin{proposition}\label{prop:rational densities}
    For all positive integers $a,b$ with $a\leq b$, there exists
    $H\in\mathcal{F}$ with $\rho(H)=b/a$.
\end{proposition}

\begin{proof}
    We use induction on $b$. If $a=b$, then the initial rooted edge
    has density $b/a=1$, as desired. Suppose that $a<b$ and let 
    $$
        d:=b-a,\qquad
        k:=\left\lceil\frac{a}{d}\right\rceil,\qquad
        c:=a-(k-1)d.
    $$
    Then we have $1\leq c\leq d<b$. By the induction hypothesis, there is
    a graph $H\in\mathcal{F}$ with $\rho(H)=d/c$.
    But Lemma~\ref{lem:recursive balancedness} then implies that
    $$
        \rho(H_k)=\frac{kd+c}{(k-1)d+c}
                 =\frac{a+d}{a}=\frac ba.
    $$
    As $H_k\in\mathcal{F}$, this completes the proof.
\end{proof}

\subsection{Proof of the rational exponents conjecture}
Now we are ready to deduce the rational exponents conjecture. We first show that, for every $H\in\mathcal{F}$ and every fixed
    positive integer $\ell$,
    \begin{equation}\label{eq:powers upper bound}
        \ex(n,H^\ell)=O\bigl(n^{2-1/\rho(H)}\bigr).
    \end{equation}
    We use induction on the number of operations used to construct
    $H$. For the initial rooted edge, $H^\ell=K_{1,\ell}$, so
    $\ex(n,H^\ell)=O(n)$, as required since $\rho(H)=1$.

    Suppose that~\eqref{eq:powers upper bound} holds for $H$ and
    let $\beta:=2-1/\rho(H)$. As $H^\ell$ is connected and
    bipartite, Theorem~\ref{thm:recursive upper bound} and
    \eqref{eq:rooted powers} give
    $$
        \ex(n,(H_k)^\ell)=\ex(n,(H^\ell)_k)
           =O\bigl(n^{1+1/(k+2-\beta)}\bigr).
    $$
    Since
    $$
        1+\frac{1}{k+2-\beta}
        =1+\frac{1}{k+1/\rho(H)}
        =2-\frac{(k-1)\rho(H)+1}{k\rho(H)+1}
        =2-\frac{1}{\rho(H_k)},
    $$
    this proves~\eqref{eq:powers upper bound} for $H_k$.

    Suppose now that $r\in[1,2)$ is rational. By
    Proposition~\ref{prop:rational densities}, there exists
    $H\in\mathcal{F}$ with $\rho(H)=1/(2-r)$.
    For every sufficiently large fixed $\ell$,
    Lemma~\ref{lem:lower bound} gives
    $\ex(n,H^\ell)=\Omega(n^r)$, while
    \eqref{eq:powers upper bound} gives
    $\ex(n,H^\ell)=O(n^r)$. Hence, 
    $\ex(n,H^\ell)=\Theta(n^r)$. Finally, non-bipartite graphs realise the
    exponent $2$. This completes the proof.

\section{The recursive upper bound} \label{sec:recursive}

\subsection{Proof outline}

In this section, we prove Theorem \ref{thm:recursive upper bound}.
The proof of this result is an adaptation of a method of Conlon, Janzer and Lee \cite{conlon2021more}. More precisely, it follows the proof of their Theorem 1.12, where they establish a tight upper bound for the extremal number of the graph obtained from a complete bipartite graph by replacing each edge by a path of length $k$ and adding a new vertex, joined to one part of the original complete bipartite graph. In both this proof and theirs, the key task is to control the number of short paths between pairs of vertices. Our host graph in Theorem \ref{thm:recursive upper bound} has roughly $n^{1+\frac{1}{k+2-\beta}}\leq n^{1+\frac{1}{k}}$ edges, so, if the graph was random-like, one would expect at most a constant number of paths of length at most $k$ between a given pair of vertices. If our host graph does indeed satisfy this ``pseudorandomness'' property, then it is not hard to prove that it contains the desired subgraph. The key lemma in \cite{conlon2021more} and also in this proof (see Lemma \ref{fewnotgood}) deals with the opposite ``structured'' case, where we do have many short paths between pairs of vertices. We now give a brief sketch of this argument.

Make the toy assumptions that for some length $2\leq \ell\leq k$, the following two properties hold: $(\star)$ between the endpoints of any path of length $\ell$, there are many other, internally vertex-disjoint paths of length $\ell$ (in particular any such path is contained in many cycles of length $2\ell$), but $(\star \star)$ there are no cycles of length less than $2\ell$ in the graph. In this case, we can find a copy of $H_k$ in our host graph as follows. Take a vertex $v\in V(G)$ and pick distinct vertices $u_1,\dots,u_t\in N(v)$, where $t$ is chosen such that $H\subset K_{t,t}$. Now take an arbitrary vertex $u^*\in N(v)$. By our assumption $(\star \star)$, if the graph is $d$-regular, then both the $(\ell-1)$th neighbourhood of $u^*$, $\Gamma_{\ell-1}(u^*)$, and the $(\ell-1)$th neighbourhood of $v$, $\Gamma_{\ell-1}(v)$, have size approximately $d^{\ell-1}$. On the other hand, our assumption $(\star)$ implies that for any path of the form $u^* v v_2 v_3 \dots v_{\ell}$ there are many internally vertex-disjoint paths of length $\ell$ between $u^*$ and $v_{\ell}$, so each vertex in $\Gamma_{\ell-1}(v)$ sends many edges to $\Gamma_{\ell-1}(u^*)$. In particular, we can find any small tree $T$ with one of its parts embedded into $\Gamma_{\ell-1}(v)$. Once such a tree is found, we can use the assumption $(\star)$ again to connect the leaves of the tree to the vertices $u_1,\dots,u_t$ by internally vertex-disjoint paths of length $\ell$. If $T$ is chosen suitably, the resulting structure will contain $H_k$ as a subgraph. More precisely, we will need to find the following tree.

\begin{definition} \label{def:tree}
	For positive integers $i$ and $t$, let $T_{i,t}$ be the tree defined by taking $t$ disjoint copies of an $(i-1)$-subdivided $t$-star and adding a new vertex adjacent to the center of each of these $t$ stars. Moreover, let $T_{0,t}$ denote the star with $t$ leaves. See Figure \ref{fig:tree} for an illustration.
\end{definition}

Note that the star with center $v$ and leaves $u_1,\dots,u_t$, together with the tree $T_{k-\ell,t}$ and disjoint paths of length $\ell$ connecting $u_1,\dots,u_t$ to the leaves of $T_{k-\ell,t}$, contains a copy of $H_k$. For an illustration, we refer the reader to Figure \ref{fig:embed F outline}, which shows the embedding in the case $H=K_{3,3}$.

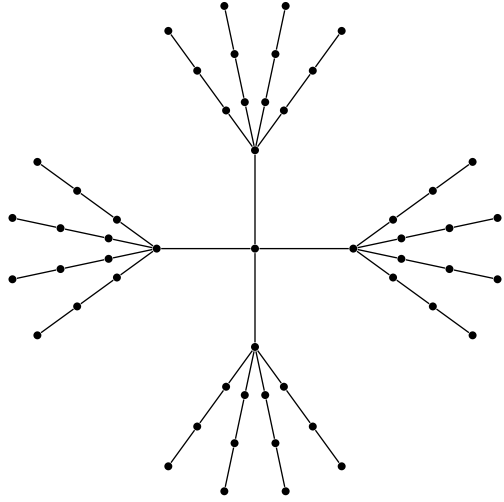
\begin{figure}
\centering
\begin{tikzpicture}[
	vertex/.style={circle,fill=black,inner sep=0pt,minimum size=3pt},
	edge/.style={draw=black,line width=0.5pt},
	font=\small
	] 
	% The additional hub and its four neighbours (the star centres).
	\node[vertex,label=below right:{}] (hub) at (0,0) {};
	\foreach \angle [count=\j] in {0,90,180,270} {
		\node[vertex] (centre-\j) at (\angle:1.3) {};
		\draw[edge] (hub) -- (centre-\j);
		
		% Four pairwise disjoint arms from this centre.
		% Levels 1 and 2 are subdivision vertices; level 3 is a leaf.
		\foreach \offset [count=\a] in {-36,-12,12,36} {
			\begin{scope}[shift={(centre-\j)},rotate=\angle+\offset]
				\foreach \level in {1,2,3} {
					\node[vertex] (v-\j-\a-\level) at ({0.65*\level},0) {};
				}
			\end{scope}
			\draw[edge] (centre-\j)
			-- (v-\j-\a-1) -- (v-\j-\a-2) -- (v-\j-\a-3);
		}
	}
\end{tikzpicture}
\caption{A copy of $T_{3,4}$.}
\label{fig:tree}
\end{figure}

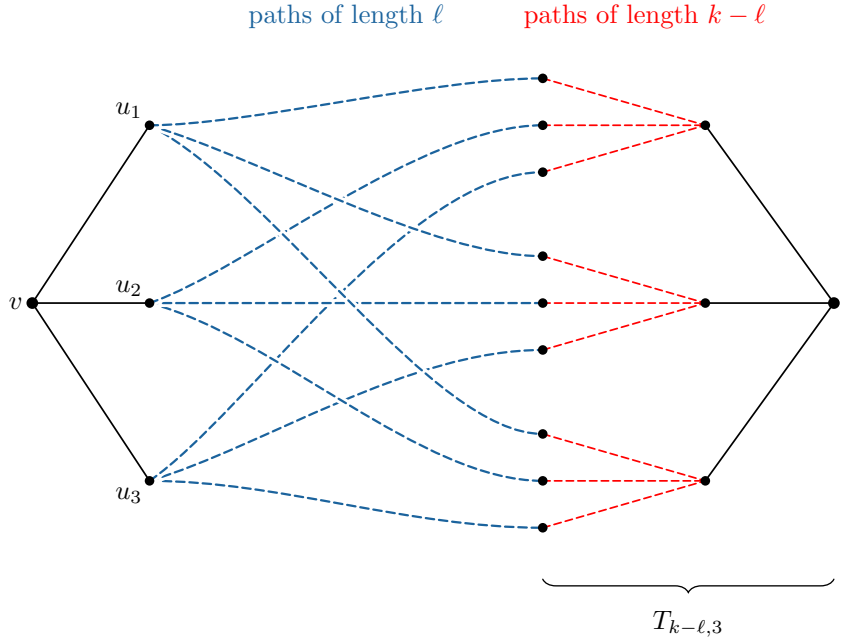
\begin{figure}
\centering
\begin{tikzpicture}[
	font=\small,
	vertex/.style={circle,fill=black,inner sep=0pt,minimum size=3.6pt},
	hub/.style={vertex,minimum size=4.4pt},
	edge/.style={draw=black,line width=.65pt},
	arm/.style={edge,draw=red,dash pattern=on 2.8pt off 1.6pt},
	added/.style={draw=addedpath,line width=.85pt,
		dash pattern=on 3.2pt off 1.8pt,
		preaction={draw=white,solid,line width=3pt}},
	every label/.style={font=\small,fill=white,inner sep=1pt},
	line cap=round
	]
	\definecolor{addedpath}{RGB}{27,99,157}
	% Main construction. The indices match r_{b,a} in the proof.
	\coordinate (v) at (0,0);
	\coordinate (w) at (10.6,0);
	\foreach \a in {1,2,3} {
		\coordinate (u-\a) at (1.55,{2.35*(2-\a)});
	}
	\foreach \b in {1,2,3} {
		\coordinate (b-\b) at (8.9,{2.35*(2-\b)});
		\foreach \a in {1,2,3} {
			\coordinate (r-\b-\a) at (6.75,{2.35*(2-\b)+.62*(2-\a)});
		}
	}
	
	% Direct hub edges and the arms of T_{i,3}.
	\foreach \a in {1,2,3} { \draw[edge] (v)--(u-\a); }
	\foreach \b in {1,2,3} {
		\draw[edge] (b-\b)--(w);
		\foreach \a in {1,2,3} { \draw[arm] (r-\b-\a)--(b-\b); }
	}
	
	% Nine pairwise internally disjoint length-ell paths. White under-strokes
	% indicate overcrossings; these are not graph vertices.
	\foreach \a in {3,2,1} {
		\foreach \b in {3,2,1} {
			\draw[added] (u-\a)
			.. controls ($(u-\a)+({1.1+.20*\b},{.45*(\a-\b)})$)
			and ($(r-\b-\a)+({-1.3-.16*\a},0)$)
			.. (r-\b-\a);
		}
	}
	
	% Place vertices and labels last so that they remain unobscured.
	\node[hub,label=left:{$v$}] at (v) {};
	\node[hub] at (w) {};
	\foreach \a in {1,2,3} {
		\ifnum\a=3
		\node[vertex,label=below left:{$u_{\a}$}] at (u-\a) {};
		\else
		\node[vertex,label=above left:{$u_{\a}$}] at (u-\a) {};
		\fi
	}
	\foreach \b in {1,2,3} {
		\ifnum\b=3
		\node[vertex] at (b-\b) {};
		\else
		\node[vertex] at (b-\b) {};
		\fi
		\foreach \a in {1,2,3} {
			\ifnum\a=3
			\node[vertex] at (r-\b-\a) {};
			\else
			\node[vertex] at (r-\b-\a) {};
			\fi
		}
	}
	
	% Equal font sizes and a common baseline for the two length labels.
	\node[anchor=base,text=addedpath,font=\small] at (4.15,3.7)
	{paths of length $\ell$};
	\node[anchor=base,text=red,font=\small] at (8.10,3.7)
	{paths of length $k-\ell$};
	
	% The bracket spans the tree, from its leaves to its additional hub w.
	\draw[decorate,decoration={brace,mirror,amplitude=5pt},line width=.5pt]
	(6.75,-3.65)--node[below=8pt] {$T_{k-\ell,3}$} (10.6,-3.65);
\end{tikzpicture}
\caption{The embedding of $(K_{3,3})_k$.}
\label{fig:embed F outline}
\end{figure}

In general, we cannot assume that the toy assumptions hold. However, as we will see, it will be sufficient to use slightly weaker versions of these assumptions to execute the embedding strategy.

\subsection{Reduction to a key lemma}

We now begin to set up the machinery, following Conlon, Janzer and Lee \cite{conlon2021more}, that will allow us to exploit a large collection of short paths between pairs of vertices, making our proof outline formal.

\begin{definition}
	Let $L$ be a positive integer. Define the function $f(\ell,L)$ for all positive integers $\ell$ recursively by setting $f(1,L)=L$ and, for $\ell\geq 2$,  $$f(\ell,L)=1+f(\ell-1,L)^{16}(\ell-1)^2\max_{1\leq i\leq \ell-1} f(i, L)f(\ell-i, L).$$
	Given this notation, we recursively define the notions of \emph{$L$-admissible} and \emph{$L$-good paths of length $\ell$} in a graph.
	Any path of length $1$ is both $L$-admissible and $L$-good. For $\ell\geq 2$, we say a path $P=v_0v_1\dots v_\ell$ is $L$-admissible if every proper subpath of $P$ is $L$-good, i.e., $v_iv_{i+1}\dots v_j$ is $L$-good for every $(i,j)\neq (0,\ell)$. A path $P$ is $L$-good if it is $L$-admissible and the number of $L$-admissible paths of length $\ell$ between $v_0$ and $v_\ell$ is at most $f(\ell,L)$.
\end{definition}

The function $f(\ell, L)$ was defined so that the following lemma holds. We note that this already appeared as {\cite[Lemma 6.3]{conlon2021more}}, though we include the short proof for completeness. 

\begin{lemma} \label{manydisjoint}
	If a path $P=v_0\dots v_\ell$ is $L$-admissible, but not $L$-good, then there exist at least $f(\ell-1,L)^{16}$ pairwise internally vertex-disjoint $L$-admissible paths of length $\ell$ from $v_0$ to $v_\ell$.
\end{lemma}

\begin{proof}
	Choose a maximal set of pairwise internally vertex-disjoint $L$-admissible paths of length $\ell$ from $v_0$ to $v_\ell$. Call them $Q_1,\dots,Q_r$ and assume that $r<f(\ell-1,L)^{16}$. Every $L$-admissible path of length $\ell$ from $v_0$ to $v_\ell$ meets one of the paths $Q_1,\dots,Q_r$ at some vertex other than $v_0$ and $v_\ell$. But $P$ is not $L$-good, so there are at least $f(\ell,L)$ such paths. By the pigeonhole principle, it follows that there exist a vertex $w$ and some $1\leq i\leq \ell-1$ such that there are at least $\frac{f(\ell,L)}{f(\ell-1,L)^{16}(\ell-1)^2}$ $L$-admissible paths $x_0x_1\dots x_\ell$ with $x_0=v_0, x_i=w, x_\ell=v_\ell$. But $\frac{f(\ell,L)}{f(\ell-1,L)^{16}(\ell-1)^2}>f(i, L)f(\ell-i, L)$, so either there are more than $f(i, L)$ $L$-good paths of length $i$ from $v_0$ to $w$ or there are more than $f(\ell-i, L)$ $L$-good paths of length $\ell-i$ from $w$ to $v_\ell$. In either case we contradict the definition of an $L$-good path.
\end{proof}

\medskip

The key lemma, which is a slight variant\footnote{The difference is that in this lemma we add two new vertices to the subdivided $K_{t,t}$, one joined to each part, whereas in \cite{conlon2021more} there was only one new vertex, joined to one of the parts.} of {\cite[Lemma 6.4]{conlon2021more}}, is as follows. Roughly speaking, it states that if a constant proportion of the paths of length $k$ in our graph contain pairs of vertices with many short paths between them, then we can find the forbidden subgraph.

\begin{lemma} \label{fewnotgood}
	Let $k\geq 2$ and $t$ be positive integers and $L$ be a sufficiently large constant in terms of $k$ and $t$. Let $G$ be a $(K_{t,t})_k$-free $d$-regular graph on $n$ vertices with $d=\omega(1)$. Then, for all sufficiently large $n$, the number of paths of length $k$ which are not $L$-good is at most $\frac{1}{10}nd^k$.
\end{lemma}

We defer the proof of Lemma~\ref{fewnotgood} to the next subsection and first show how to conclude the proof of Theorem~\ref{thm:recursive upper bound}.

We start with the following immediate corollary of Lemma \ref{fewnotgood}.

\begin{lemma} \label{lem:manygood}
    Let $H$ be a connected bipartite graph, $k$ be a positive integer and $L$ be a sufficiently large constant in terms of $|H|$ and $k$. Let $G$ be an $H_k$-free $d$-regular graph on $n$ vertices with $d=\omega(1)$. Then, for all sufficiently large $n$, the number of paths of length $k+2$ in which every subpath of length $k$ is $L$-good is at least $\frac{1}{10}nd^{k+2}$.
\end{lemma}

\begin{proof}
    The statement is trivial if $k=1$ since every path of length $1$ is $L$-good, so we may assume that $k\geq 2$. Choose $t=|H|$ and note that then $G$ is $(K_{t,t})_k$-free. By Lemma \ref{fewnotgood}, for large enough $L$, the number of paths of length $k$ in $G$ which are not $L$-good is therefore at most $\frac{1}{10}nd^k$. As $G$ is $d$-regular, any path of length $k$ can be extended in at most $3d^2$ ways to a path of length $k+2$, where the factor $3$ comes from choosing the position of the subpath of length $k$. Hence, the number of paths of length $k+2$ with a non-good subpath of length $k$ is at most $\frac{3}{10}nd^{k+2}$. Since the total number of paths of length $k+2$ in $G$ is at least $\frac{2}{5}nd^{k+2}$, the lemma follows.
\end{proof}

The goodness condition bounds the number of admissible paths of a given length between the endpoints of each subpath. We now pass to a collection in which every subpath of length at most $k$ is uniquely determined by its endpoints. This property is captured by the following definition.

\begin{definition}
    We call a collection $\mathcal{P}$ of paths of length $k+2$ in a graph $G$ \emph{diverse} if there exists a partition $V(G)=V_0\cup V_1\cup \dots \cup V_{k+2}$ such that each $P\in \mathcal{P}$ is of the form $v_0v_1\dots v_{k+2}$ for some $v_i\in V_i$ and, for any distinct $u_0\dots u_{k+2},v_0\dots v_{k+2}\in \mathcal{P}$ and indices $0\leq i<j\leq k+2$ with $j-i\leq k$, if $u_i=v_i$ and $u_j=v_j$, then $u_q=v_q$ for all $i\leq q\leq j$. In other words, if two paths intersect twice at distance at most $k$, then they induce the same subpath between these intersections.
\end{definition}

\begin{lemma} \label{lem:obtain diverse}
    For every positive integer $k$ and $L$, there is a positive constant $c$ such that if $\mathcal{P}_0$ is a collection of paths of length $k+2$ in which every subpath of length $k$ is $L$-good, then there is a subcollection $\mathcal{P}\subset \mathcal{P}_0$ of size at least $c|\mathcal{P}_0|$ which is diverse.
\end{lemma}

\begin{proof}
    Consider a random partition $V_0\cup \dots \cup V_{k+2}$ of the vertex set of the graph. Note that the probability a given path $v_0 \dots v_{k+2}$ satisfies that $v_i\in V_i$ for every $i$ is at least some positive constant $c_0$ depending on $k$. Let $\mathcal{P}_1\subset \mathcal{P}_0$ be the subset of paths satisfying this property. Clearly, we may fix a choice of the partition for which $|\mathcal{P}_1|\geq c_0|\mathcal{P}_0|$. Note that for each $P\in \mathcal{P}_1$, every subpath of $P$ of length $\ell\leq k$ is $L$-good, meaning that for any such subpath, there are at most $f(\ell,L)$ $L$-admissible paths of length $\ell$ between its endpoints. It follows that for each $0\leq i<j\leq k+2$ with $j-i\leq k$, $v_i\in V_i$ and $v_j\in V_j$, there are at most $f(j-i,L)$ different paths between $v_i$ and $v_j$ appearing as the intersection of a member of $\mathcal{P}_1$ with $V_i\cup \dots \cup V_j$. For each pair of vertices $(v_i,v_j)\in V_i\times V_j$, sample one such path independently and uniformly at random (if one exists) and call the sampled paths \emph{allowed}. Now let $\mathcal{P}\subset \mathcal{P}_1$ be the subcollection of paths in which every subpath of length at most $k$ is allowed. Note that each $P\in \mathcal{P}_1$ has $O_k(1)$ subpaths of length at most $k$ and each of them is allowed independently with probability at least $1/f(k,L)$. Hence, the probability that a given $P\in \mathcal{P}_1$ remains in $\mathcal{P}$ is at least some positive constant depending on $k$ and $L$. Thus, with positive probability, we have $|\mathcal{P}|\geq c|\mathcal{P}_0|$ for some positive constant $c$ depending on $k$ and $L$. Since $\mathcal{P}$ is diverse (with respect to the partition $V_0\cup \dots \cup V_{k+2}$), the proof of the lemma is complete.
\end{proof}

Combining Lemmas \ref{lem:manygood} and \ref{lem:obtain diverse}, we can conclude the existence of a large diverse collection of paths.

\begin{lemma} \label{lem:many diverse}
    Let $H$ be a connected bipartite graph and $k$ be a positive integer. Let $G$ be an $H_k$-free $d$-regular graph on $n$ vertices with $d=\omega(1)$. Then, for all sufficiently large $n$, there is a diverse collection $\mathcal{P}$ of paths of length $k+2$ with $|\mathcal{P}|\geq cnd^{k+2}$, where $c$ is a positive constant depending only on $|H|$ and $k$.
\end{lemma}

The last ingredient in the proof of Theorem \ref{thm:recursive upper bound} is the following result of Chakraborti, Janzer, Methuku and Montgomery \cite{chakraborti2026regular} that allows us to assume that our host graph is regular. We remark that one could instead rely on a more classical almost-regularization lemma of Erd\H os and Simonovits \cite{erdossim70,jiang2012turan}, but this would lead to slightly more complicated statements in our technical lemmas.

\begin{lemma}[Chakraborti--Janzer--Methuku--Montgomery \cite{chakraborti2026regular}]\label{regularization}
    For any real $0<\varepsilon < 1$, there exists $\delta > 0$ such that the following holds. Suppose $c > 0$ and $n$ is a positive integer that is sufficiently large in terms of $\varepsilon$ and $c$. If $G$ is a graph on $n$ vertices with $e(G) \geq c n^{ 1 + \varepsilon}$, then $G$ contains a regular subgraph $G'$ on $m \geq n^{\delta}$ vertices such that $e(G') \geq \delta c m^{ 1 + \varepsilon}$. 
\end{lemma}

We are now ready to prove Theorem \ref{thm:recursive upper bound}. Given all this preparation, the proof of the theorem now resembles (and generalizes) the classical Erd\H os--Simonovits result mentioned in the introduction that they developed in order to give an upper bound on the extremal number of the cube.

\begin{proof}[Proof of Theorem \ref{thm:recursive upper bound}]
	Let $H$ be a connected bipartite graph with $\ex(n, H) =  O(n^{\beta})$. Let $G$ be an $n$-vertex graph with $\omega(n^{1 + \frac{1}{k + 2 - \beta}})$ edges. We seek to show that $G$ contains a copy of $H_k$, which will then imply that $\ex(n, H_k) = O(n^{ 1 + \frac{1}{k + 2 - \beta}})$. Suppose on the contrary that $G$ is $H_k$-free. 

    By Lemma~\ref{regularization}, $G$ has a $d$-regular subgraph $G'$ on $m$ vertices, where $m=\omega(1)$ and $d = \omega(m^{\frac{1}{k + 2 - \beta}})$. Since $G$ is $H_k$-free, so is $G'$ and hence we may apply Lemma~\ref{lem:many diverse} to find a diverse collection $\mathcal{P}$ of paths in $G'$ with $|\mathcal{P}| \geq cm d^{ k + 2}$, where $c$ is a positive constant depending only on $|H|$ and $k$. Let $V_0\cup \dots \cup V_{k+2}$ be the partition of $V(G')$ witnessing the diversity of $\mathcal{P}$. 

    By the pigeonhole principle, there exist vertices $u\in V_0$ and $v\in V_{k+2}$ such that there are at least $\frac{cm d^{k + 2}}{m^2} = \omega(d^{\beta})$ members of $\mathcal{P}$ with endpoints $u$ and $v$. Call the collection of these paths $\mathcal{P}_{uv}$. 

    Let $F$ be the bipartite graph with parts $N_{G'}(u)\cap V_1$ and $N_{G'}(v)\cap V_{k+1}$ where for $x \in N_{G'}(u)\cap V_1$ and $y \in N_{G'}(v)\cap V_{k+1}$, we have $xy\in E(F)$ if and only if there is a path $w_0w_1\dots w_{k + 1}w_{k + 2}$ in $\mathcal{P}_{uv}$ with $w_0 = u$, $w_1 = x$, $w_{k + 1} = y$ and $w_{k + 2} = v$ or, in other words, if there is a path in $\mathcal{P}_{uv}$ starting with $ux$ and ending with $yv$. Observe that, by the diversity of the collection $\mathcal{P}_{uv}$, there is at most one path in $\mathcal{P}_{uv}$ starting with $ux$ and ending with $yv$. Hence, we have $e(F)=|\mathcal{P}_{uv}|= \omega(d^{\beta})$. Moreover, note that if distinct paths $uw_1 \dots w_{k + 1}v$ and $uz_1\dots z_{ k + 1}v$ in $\mathcal{P}_{uv}$ intersect internally, then their only internal intersection is either $w_1=z_1$ or $w_{k+1}=z_{k+1}$. Indeed, since $w_i,z_i\in V_i$, any intersection is of the form $w_i=z_i$ for some $1\leq i\leq k+1$. Suppose that $w_i=z_i$ for some $2\leq i\leq k$. Then, by diversity, as $i \leq k$ and the first vertex of these paths is the same, we must have $w_j=z_j$ for all $1\leq j \leq i$. Similarly, as $i \geq 2$ and the last vertex of these paths is the same, we must have $w_j=z_j$ for all $i\leq j\leq k+1$. Thus, these two paths are identical.

    Note that  we have $v(F)  \leq 2 d$. Since $e(F)=\omega(d^{\beta})$ and $\ex(n,H)=O(n^{\beta})$, it follows that $F$ contains $H$ as a subgraph. Let $J$ be this copy of $H$ in $F$. Replacing each edge $xy\in E(J)$ with the unique path $xw_2 \dots w_k y$ in $G'$ such that $uxw_2\dots w_k yv\in \mathcal{P}_{uv}$ and adding the vertices $u$ and $v$, we obtain a copy of $H_k$ in $G'$, which is nondegenerate by the previous paragraph.
\end{proof}

\subsection{The proof of Lemma \ref{fewnotgood}}

In this subsection we prove Lemma \ref{fewnotgood}. We will need the following definition.

\begin{definition}
	A pair of distinct vertices $\{x, y\}$ in $G$ is said to be $(\ell,L)$-\emph{heavy} for some $\ell\geq 2$ if there are at least $f(\ell-1,L)^{16}$ internally vertex-disjoint paths of length $\ell$ from $x$ to $y$. In particular, Lemma~\ref{manydisjoint} implies that if there is an $L$-admissible, but not $L$-good, path of length $\ell$ from $x$ to $y$, then $\{x,y\}$ is $(\ell,L)$-heavy.
\end{definition}

In what follows, for $v\in V(G)$, we shall write $\Gamma_i(v)$ for the set of vertices $u\in V(G)$ for which there exists a path of length $i$ from $v$ to $u$. The next lemma, a variant of {\cite[Lemma 6.7]{conlon2021more}}, will be applied with $X=N(v)$ and $Y=\Gamma_{\ell-1}(v)$ to show that in a $(K_{t,t})_k$-free graph there cannot be many heavy pairs between $N(v)$ and $\Gamma_{\ell-1}(v)$. We remind the reader that the tree $T_{i,t}$ was introduced in Definition \ref{def:tree}.

\begin{lemma} \label{findstructure}
	Let $\ell\geq 2$, $i\geq 0$ and $t\geq 2$ be integers. Let $G$ be a $d$-regular graph on $n$ vertices with $d=\omega(1)$. Let $v\in V(G)$ and let $X,Y\subset V(G)$ be such that $|X|=\omega(1)$, $|Y|\geq \frac{d^{\ell-1}}{f(\ell-1,L)^2}$ and the number of $(x,y)\in X\times Y$ such that $(x,y)$ is $(\ell,L)$-heavy is at least $\frac{|X||Y|}{f(\ell-1,L)^2}$. Then, provided that $L$ is sufficiently large in terms of $\ell$, $i$ and $t$, there exist a copy of $T_{i,t}$ in $G-v$ with its leaves embedded as a set $R\subset Y$ and a subset $X'\subset X$ such that $|X'|=\omega(1)$ and $(x',r)$ is $(\ell,L)$-heavy for every $x'\in X'$ and $r\in R$.
\end{lemma}

\begin{proof}
	Let us write $f:=f(\ell-1,L)$. Let $Y'$ be the set of those $y\in Y$ for which the number of $x\in X$ such that $(x,y)$ is $(\ell,L)$-heavy is at least $\frac{|X|}{2f^2}$. Then $|Y'|\geq \frac{|Y|}{2f^2}\geq \frac{d^{\ell-1}}{2f^4}$ and the number of $(x,y)\in X\times Y'$ which are $(\ell,L)$-heavy is at least $\frac{|X||Y'|}{2f^2}$. In particular, there exists some $x^*\in X$ such that there are at least $\frac{|Y'|}{2f^2}$ $(\ell,L)$-heavy pairs of the form $(x^*,y)$ for $y\in Y'$. Let $U_0\subset Y'$ be the set of such vertices $y$. Note that $|U_0|\geq \frac{d^{\ell-1}}{4f^6}$.
	
	Let $W_0=\Gamma_{\ell-1}(x^*)$. Since $G$ is $d$-regular, we have $|W_0|\leq d^{\ell-1}$. Since $(x^*,u)$ is $(\ell,L)$-heavy for all $u\in U_0$, there are at least $f^{16}$ internally vertex-disjoint paths of length $\ell$ from $u$ to $x^*$. The second vertex on each of these paths is in $W_0$, so there are at least $|U_0|f^{16}$ pairs $(u,w)\in U_0\times W_0$ with $uw\in E(G)$. Since $|U_0|\geq \frac{d^{\ell-1}}{4f^6}$ and $|W_0|\leq d^{\ell-1}$, it follows that the number of pairs $(u,w)\in U_0\times W_0$ with $uw\in E(G)$ is also at least $|W_0|f^{10}/4$. Note that $U_0$ and $W_0$ may not be disjoint. However, we can easily pass to disjoint subsets $U\subset U_0$ and $W\subset W_0$ such that the number of pairs $(u,w)\in U\times W$ with $uw\in E(G)$ is at least $2(|U|+|W|)f^9$. Let $G'$ be the bipartite subgraph of $G$ with parts $U$ and $W$ and edges inherited from $G$. By the above discussion, $G'$ has average degree at least $f^9$.
    
    Let $t_0=f^{4}$. Note that $T_{i,t_0}$ has $it_0^2+t_0+1$ vertices. Hence, if $L$ is sufficiently large, then the average degree condition of $G'$ implies that $G'-v$ contains a copy of $T_{i,t_0}$ with the set of leaves embedded in $U$ as some subset $R_0\subset U$. Let us call the embedding of this tree $T$. Since $U\subset Y'$, for each $u\in R_0$, the number of $x\in X$ for which $(x,u)$ is $(\ell,L)$-heavy is at least $\frac{|X|}{2f^2}$. Writing $\widetilde{X}$ for the set of $x\in X$ for which there are at least $\frac{|R_0|}{4f^2}$ vertices $u\in R_0$ such that $(x,u)$ is $(\ell,L)$-heavy, we have $|\widetilde{X}|\geq \frac{|X|}{4f^2}$, so $|\widetilde{X}|=\omega(1)$. If $i\geq 1$, then $\frac{|R_0|}{4f^2} = \frac{t_0^2}{4f^2}\geq 2t_0t$ and it follows that for each $x\in \widetilde{X}$ there are at least $t$ constituent stars in $T$ containing at least $t$ leaves $u$ such that $(x,u)$ is heavy. Hence, for each $x\in \widetilde{X}$, there exists a subtree $T(x)\subset T$ isomorphic to $T_{i,t}$ such that $(x,u)$ is $(\ell,L)$-heavy for each leaf $u$ of $T(x)$. The same conclusion holds also when $i=0$. Since $|\widetilde{X}|=\omega(1)$ and $T$ has constant size, there must exist a set $X'\subset \widetilde{X}$ of size $\omega(1)$ such that $T(x)$ is the same for all $x\in X'$. Then the embedding of $T_{i,t}$ in $G$ corresponding to this tree $T(x)$ satisfies the conditions of the lemma.
\end{proof}

Following {\cite[Lemma 6.8]{conlon2021more}}, we now prove that only a small proportion of the paths of given length through a fixed vertex are admissible but not good. 

\begin{lemma} \label{lem:per vertex}
	Let $2\leq \ell\leq k$ and $t$ be positive integers and $L$ be a sufficiently large constant in terms of $k$ and $t$. Let $G$ be a $(K_{t,t})_k$-free $d$-regular graph on $n$ vertices with $d=\omega(1)$ and fix some $v\in V(G)$. Then, for all sufficiently large $n$, the number of $L$-admissible, but not $L$-good, paths of the form $v_0vv_2v_3\dots v_\ell$ is at most $\frac{d^\ell}{f(\ell-1,L)}$.
\end{lemma}

\begin{proof}
	Suppose that there are more than $\frac{d^\ell}{f(\ell-1,L)}$ such paths. Note that if the path $v_0vv_2\dots v_\ell$ is $L$-admissible, then $vv_2\dots v_\ell$ is $L$-good, so the number of $L$-admissible paths of length $\ell-1$ from $v$ to $v_\ell$ is at most $f(\ell-1,L)$. Hence, the number of $L$-admissible paths $v_0vv_2\dots v_\ell$ is at most $f(\ell-1,L)$ for any fixed $v_0$ and $v_\ell$. Thus, the number of pairs $(x,y)\in N(v)\times \Gamma_{\ell-1}(v)$ such that there is an $L$-admissible, but not $L$-good, path of the form $xvv_2\dots v_{\ell-1}y$ is at least $\frac{d^\ell}{f(\ell-1,L)^2}\geq \frac{|N(v)||\Gamma_{\ell-1}(v)|}{f(\ell-1,L)^2}$. Recall that any such pair $(x,y)$ is $(\ell,L)$-heavy. Note that we must also have $\frac{d^\ell}{f(\ell-1,L)^2}\leq |N(v)||\Gamma_{\ell-1}(v)|$, so $|\Gamma_{\ell-1}(v)|\geq \frac{d^{\ell-1}}{f(\ell-1,L)^2}$. Hence, the conditions of Lemma \ref{findstructure} are satisfied with $X=N(v)$ and $Y=\Gamma_{\ell-1}(v)$.
	
	Our aim now is to find a copy of $(K_{t,t})_k$ in $G$, which will yield a contradiction. By Lemma \ref{findstructure} (applied with $i=k-\ell$), there exist a copy $T$ of the tree $T_{k-\ell,t}$ in $G-v$ and a set $X'$ of size $\omega(1)$ such that $(x',r)$ is $(\ell,L)$-heavy for every $x'\in X'$ and every leaf $r$ of $T$. Choose vertices $u_1,\dots,u_t$ in $X'$ that are distinct from each other, from $v$ and from the vertices of $T$. To complete the embedding of $(K_{t,t})_k$, we consider two cases.
	
	If $\ell=k$, then $i=0$ and $T$ is a star with $t$ leaves. Call these leaves $r_1,\dots,r_t$. Since $(u_a,r_b)$ is $(\ell,L)$-heavy for all $a,b\in [t]$ and $L$ is sufficiently large, there exist internally vertex-disjoint paths of length $\ell=k$ connecting the pairs $(u_a,r_b)$ in such a way that these connecting paths are internally disjoint from $v$, $\{u_a:a\in [t]\}$ and the vertex set of $T$. Then, since $u_1,\dots,u_t\in N(v)$, these connecting paths, together with $v$, $\{u_a:a\in [t]\}$ and $T$, form a copy of $(K_{t,t})_k$, as desired.
	
	Assume now that $\ell<k$. Let $\{r_{b,c}: b,c\in [t]\}$ be the leaves of $T$, where, for any $b\in [t]$, the set of leaves $\{r_{b,c}: c\in [t]\}$ belong to the same $(i-1)$-subdivided $t$-star with center $w_b$. Since $(u_a,r_{b,c})$ is $(\ell,L)$-heavy for all $a,b,c\in [t]$ and $L$ is sufficiently large, there exist internally vertex-disjoint paths of length $\ell=k-i$ connecting the pairs $(u_a,r_{b,a})$ for all $a,b\in [t]$ in such a way that these connecting paths are internally disjoint from $v$, $\{u_a:a\in [t]\}$ and the vertex set of $T$. Once again, these connecting paths, together with $v$, $\{u_a:a\in [t]\}$ and $T$, form a copy of $(K_{t,t})_k$, as desired. Indeed, between each $u_a$ and $w_b$ there is a path of length $k$ through $r_{b,a}$ such that all these paths are internally vertex-disjoint. Hence, we have a $(k-1)$-subdivided $K_{t,t}$, together with a vertex $v$ joined to $u_1,\dots,u_t$ and the center of $T$ joined to $w_1,\dots,w_t$. See Figure \ref{fig:embed F} for an illustration.
\end{proof}

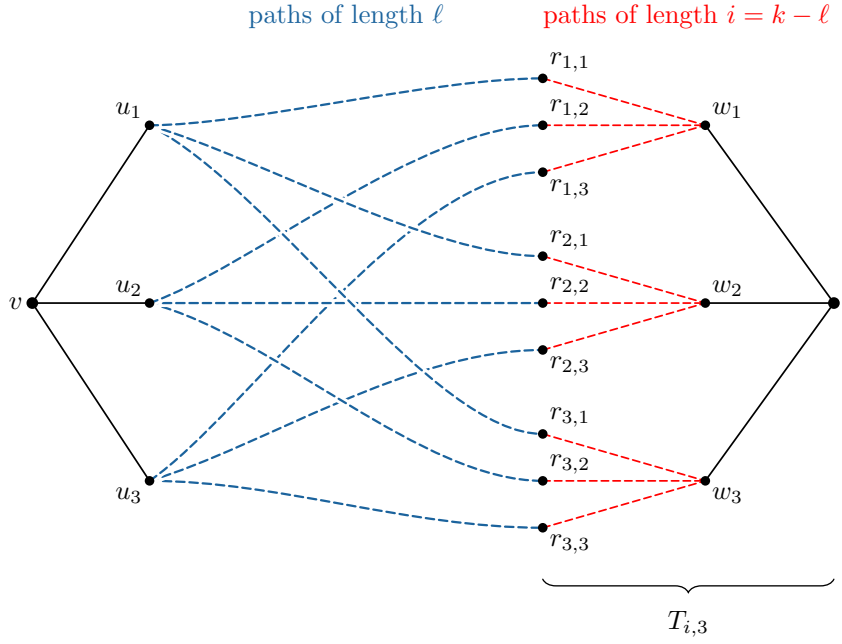
\begin{figure}
\centering
\begin{tikzpicture}[
	font=\small,
	vertex/.style={circle,fill=black,inner sep=0pt,minimum size=3.6pt},
	hub/.style={vertex,minimum size=4.4pt},
	edge/.style={draw=black,line width=.65pt},
	arm/.style={edge,draw=red,dash pattern=on 2.8pt off 1.6pt},
	added/.style={draw=addedpath,line width=.85pt,
		dash pattern=on 3.2pt off 1.8pt,
		preaction={draw=white,solid,line width=3pt}},
	every label/.style={font=\small,fill=white,inner sep=1pt},
	line cap=round
	]
	\definecolor{addedpath}{RGB}{27,99,157}
	% Main construction. The indices match r_{b,a} in the proof.
	\coordinate (v) at (0,0);
	\coordinate (w) at (10.6,0);
	\foreach \a in {1,2,3} {
		\coordinate (u-\a) at (1.55,{2.35*(2-\a)});
	}
	\foreach \b in {1,2,3} {
		\coordinate (b-\b) at (8.9,{2.35*(2-\b)});
		\foreach \a in {1,2,3} {
			\coordinate (r-\b-\a) at (6.75,{2.35*(2-\b)+.62*(2-\a)});
		}
	}
	
	% Direct hub edges and the arms of T_{i,3}.
	\foreach \a in {1,2,3} { \draw[edge] (v)--(u-\a); }
	\foreach \b in {1,2,3} {
		\draw[edge] (b-\b)--(w);
		\foreach \a in {1,2,3} { \draw[arm] (r-\b-\a)--(b-\b); }
	}
	
	% Nine pairwise internally disjoint length-ell paths. White under-strokes
	% indicate overcrossings; these are not graph vertices.
	\foreach \a in {3,2,1} {
		\foreach \b in {3,2,1} {
			\draw[added] (u-\a)
			.. controls ($(u-\a)+({1.1+.20*\b},{.45*(\a-\b)})$)
			and ($(r-\b-\a)+({-1.3-.16*\a},0)$)
			.. (r-\b-\a);
		}
	}
	
	% Place vertices and labels last so that they remain unobscured.
	\node[hub,label=left:{$v$}] at (v) {};
	\node[hub] at (w) {};
	\foreach \a in {1,2,3} {
		\ifnum\a=3
		\node[vertex,label=below left:{$u_{\a}$}] at (u-\a) {};
		\else
		\node[vertex,label=above left:{$u_{\a}$}] at (u-\a) {};
		\fi
	}
	\foreach \b in {1,2,3} {
		\ifnum\b=3
		\node[vertex,label=below right:{$w_{\b}$}] at (b-\b) {};
		\else
		\node[vertex,label=above right:{$w_{\b}$}] at (b-\b) {};
		\fi
		\foreach \a in {1,2,3} {
			\ifnum\a=3
			\node[vertex,label=below right:{$r_{\b,\a}$}] at (r-\b-\a) {};
			\else
			\node[vertex,label=above right:{$r_{\b,\a}$}] at (r-\b-\a) {};
			\fi
		}
	}
	
	% Equal font sizes and a common baseline for the two length labels.
	\node[anchor=base,text=addedpath,font=\small] at (4.15,3.7)
	{paths of length $\ell$};
	\node[anchor=base,text=red,font=\small] at (8.65,3.7)
	{paths of length $i=k-\ell$};
	
	% The bracket spans the tree, from its leaves to its additional hub w.
	\draw[decorate,decoration={brace,mirror,amplitude=5pt},line width=.5pt]
	(6.75,-3.65)--node[below=8pt] {$T_{i,3}$} (10.6,-3.65);
\end{tikzpicture}
\caption{The embedding of $(K_{t,t})_k$ for $t=3$ and $\ell<k$.}
\label{fig:embed F}
\end{figure}

Considering that there are $n$ choices for the vertex $v$, Lemma \ref{lem:per vertex} immediately implies the following global upper bound on the number of $L$-admissible but not $L$-good paths of length $\ell$.

\begin{corollary} \label{admissiblenotgood}
	Let $2\leq \ell\leq k$ and $t$ be positive integers and $L$ be a sufficiently large constant in terms of $k$ and $t$. Let $G$ be a $(K_{t,t})_k$-free $d$-regular graph on $n$ vertices with $d=\omega(1)$. Then, for all sufficiently large $n$, the number of $L$-admissible, but not $L$-good, paths of length $\ell$ in $G$ is at most $\frac{nd^\ell}{f(\ell-1,L)}$.
\end{corollary}

It is now easy to deduce Lemma \ref{fewnotgood}.

\begin{proof}[\textnormal{\textbf{Proof of Lemma \ref{fewnotgood}}}]
	Suppose that the path $u_0u_1\dots u_k$ is not $L$-good. Take $0\leq i<j\leq k$ with $j-i$ minimal such that $u_iu_{i+1}\dots u_j$ is not $L$-good. Then $u_i\dots u_j$ is $L$-admissible. For any fixed $i,j$, by Corollary \ref{admissiblenotgood}, the number of such paths is at most $\frac{nd^{j-i}}{f(j-i-1,L)}\cdot d^{k-(j-i)}=\frac{1}{f(j-i-1,L)}nd^k\leq \frac{1}{L}nd^k$. Using that $i$ and $j$ can take at most $k+1$ values each, it follows that the number of not $L$-good paths of length $k$ is at most $\frac{(k+1)^2}{L}nd^k\leq \frac{1}{10}nd^k$.
\end{proof}

\section{Concluding remarks}

While the rational exponents conjecture is now resolved, several related problems remain open. These include the Bukh--Conlon conjecture on rooted powers of trees and the Kang--Kim--Liu subdivision conjecture, both of which were known to imply the rational exponents conjecture.

\begin{conjecture}[Bukh--Conlon \cite{bukhconlon18}]\label{conj:bukh-conlon}
    For any balanced rooted tree $T$ and any positive integer $\ell$, $\ex(n,T^{\ell})=O_{\ell}(n^{2-1/\rho(T)})$.
\end{conjecture}

\begin{conjecture}[Kang--Kim--Liu \cite{kkl21}] \label{conj:KKL}
    Let $F$ be a bipartite graph with $\ex(n,F)=O(n^{1+\alpha})$ for some $\alpha>0$. Then $\ex(n,F')=O(n^{1+\alpha/2})$, where $F'$ is the $1$-subdivision of $F$.
\end{conjecture}

Progress on Conjecture~\ref{conj:bukh-conlon} can be found in~\cite{conjan22,conlon2021more,janzer2020extremal,jjm22,jly26,jmy22,jiangqiu20, jiangqiu23,kkl21}, while 
progress on Conjecture~\ref{conj:KKL} can be found in \cite{conlon2021more,jly26}.

\medskip

\paragraph{Statement on AI use.} We used ChatGPT Pro to draw the figures in this paper and to proofread an early draft.

\bibliographystyle{abbrv}
\bibliography{bibliography}

\begin{thebibliography}{10}

\bibitem{bloom26}
T.~F. Bloom.
\newblock Webpage available at https://www.erdosproblems.com/571, 2026.

\bibitem{bukh15}
B.~Bukh.
\newblock Random algebraic construction of extremal graphs.
\newblock {\em Bull. Lond. Math. Soc.}, 47(6):939--945, 2015.

\bibitem{bukhconlon18}
B.~Bukh and D.~Conlon.
\newblock Rational exponents in extremal graph theory.
\newblock {\em J. Eur. Math. Soc. (JEMS)}, 20(7):1747--1757, 2018.

\bibitem{chakraborti2026regular}
D.~Chakraborti, O.~Janzer, A.~Methuku, and R.~Montgomery.
\newblock Regular subgraphs at every density.
\newblock {\em Trans. Amer. Math. Soc.}, 379(11):8069--8090, 2026.

\bibitem{conjan22}
D.~Conlon and O.~Janzer.
\newblock Rational exponents near two.
\newblock {\em Adv. Comb.}, 2022:Paper No. 9, 10 pp., 2022.

\bibitem{conlon2021more}
D.~Conlon, O.~Janzer, and J.~Lee.
\newblock More on the extremal number of subdivisions.
\newblock {\em Combinatorica}, 41(4):465--494, 2021.

\bibitem{erdos81}
P.~Erd\H{o}s.
\newblock On the combinatorial problems which {I} would most like to see
  solved.
\newblock {\em Combinatorica}, 1(1):25--42, 1981.

\bibitem{erdossim66}
P.~Erd\H{o}s and M.~Simonovits.
\newblock A limit theorem in graph theory.
\newblock {\em Studia Sci. Math. Hungar.}, 1:51--57, 1966.

\bibitem{erdossim70}
P.~Erd\H{o}s and M.~Simonovits.
\newblock Some extremal problems in graph theory.
\newblock In {\em Combinatorial theory and its applications, {I}-{III} ({P}roc.
  {C}olloq., {B}alatonf\"ured, 1969)}, volume~4 of {\em Colloq. Math. Soc.
  J\'anos Bolyai}, pages 377--390. North-Holland, Amsterdam-London, 1970.

\bibitem{erdosstone46}
P.~Erd\H{o}s and A.~H. Stone.
\newblock On the structure of linear graphs.
\newblock {\em Bull. Amer. Math. Soc.}, 52:1087--1091, 1946.

\bibitem{janzer2020extremal}
O.~Janzer.
\newblock The extremal number of the subdivisions of the complete bipartite
  graph.
\newblock {\em SIAM J. Discrete Math.}, 34(1):241--250, 2020.

\bibitem{jjm22}
T.~Jiang, Z.~Jiang, and J.~Ma.
\newblock Negligible obstructions and {T}ur\'an exponents.
\newblock {\em Ann. Appl. Math.}, 38(3):356--384, 2022.

\bibitem{jly26}
T.~Jiang, S.~Longbrake, and L.~Yepremyan.
\newblock Rational exponents near 3/2.
\newblock {\em Preprint available at arXiv:2607.19607 [math.CO]}, 2026.

\bibitem{jmy22}
T.~Jiang, J.~Ma, and L.~Yepremyan.
\newblock On {T}ur\'an exponents of bipartite graphs.
\newblock {\em Combin. Probab. Comput.}, 31(2):333--344, 2022.

\bibitem{jiangqiu20}
T.~Jiang and Y.~Qiu.
\newblock Tur\'an numbers of bipartite subdivisions.
\newblock {\em SIAM J. Discrete Math.}, 34(1):556--570, 2020.

\bibitem{jiangqiu23}
T.~Jiang and Y.~Qiu.
\newblock Many {T}ur\'an exponents via subdivisions.
\newblock {\em Combin. Probab. Comput.}, 32(1):134--150, 2023.

\bibitem{jiang2012turan}
T.~Jiang and R.~Seiver.
\newblock Tur{\'a}n numbers of subdivided graphs.
\newblock {\em SIAM J. Discrete Math.}, 26(3):1238--1255, 2012.

\bibitem{kkl21}
D.~Y. Kang, J.~Kim, and H.~Liu.
\newblock On the rational {T}ur\'an exponents conjecture.
\newblock {\em J. Combin. Theory Ser. B}, 148:149--172, 2021.

\bibitem{kst54}
T.~K\H{o}v\'{a}ri, V.~T. S\'os, and P.~Tur\'an.
\newblock On a problem of {K}. {Z}arankiewicz.
\newblock {\em Colloq. Math.}, 3:50--57, 1954.

\end{thebibliography}

\end{document}